\documentclass[10pt, english]{elsarticle}
\usepackage{amsthm}
\usepackage{amsmath}
\usepackage{latexsym, amssymb}
\usepackage{txfonts}
\usepackage{mathtools}
\usepackage{color}
\usepackage{babel}
\usepackage[all]{xy}
\usepackage[capitalise]{cleveref}

\newtheorem{thm}{Theorem}[section] %the resolution could also be [subsection]

\newtheorem{cor}[thm]{Corollary}

\newtheorem{lem}[thm]{Lemma}
\newtheorem{prop}[thm]{Proposition}

\newtheorem{ques}[thm]{Question}

\theoremstyle{definition}
\newtheorem{rem}[thm]{Remark}
\newtheorem{exmpl}[thm]{Example}

\newcommand\operA[2]{{\if!#2!\operatorname{#1}\else{\operatorname{#1}_{#2}^{\phantom{I}}}\fi}} % To be used within Bdefs. Usage: $\operA{N}{K/F}$ produces $N_{K/F}$; $\operA{N}{}$ produces $N$.

\newcommand{\Trace}[1][]{\if!#1!\operatorname{Tr}\else{\operatorname{Tr}_{#1}^{\phantom{I}}}\fi} % Usage: $\Tr[K/F](a)$.

\long\def\forget#1\forgotten{{}} %

\def\({\left(}
\def\){\right)}

\newcommand\LAY[3][]{{\begin{array}{c}\mbox{#2} \if#1!{}\else{+}\fi \\ \mbox{#3}\end{array}}}

\makeatletter
\newcommand{\bigperp}{%
  \mathop{\mathpalette\bigp@rp\relax}%
  \displaylimits
}

\newcommand{\bigp@rp}[2]{%
  \vcenter{
    \m@th\hbox{\scalebox{\ifx#1\displaystyle2.1\else1.5\fi}{$#1\perp$}}
  }%
}
\makeatother

\newcommand{\qf}[1]{\mbox{$\langle #1\rangle $}}

\renewcommand{\geq}{\geqslant}
\renewcommand{\leq}{\leqslant}

\makeatletter
\def\ps@pprintTitle{%
 \let\@oddhead\@empty
 \let\@evenhead\@empty
 \def\@oddfoot{\centerline{\thepage}}%
 \let\@evenfoot\@oddfoot}
\makeatother

\newif\iffurther
\furtherfalse
\input xy
\input xyidioms.tex
\usepackage{xy}
\xyoption{all} %
\begin{document}
\begin{frontmatter}

\title{Types of Linkage of Quadratic Pfister Forms}

\author{Adam Chapman}
\ead{adam1chapman@yahoo.com}
\address{Department of Computer Science, Tel-Hai Academic College, Upper Galilee, 12208 Israel}
\author{Andrew Dolphin}
\ead{Andrew.Dolphin@uantwerpen.be}
\address{Departement Wiskunde--Informatica, Universiteit Antwerpen, Belgium}

\begin{abstract}
Given a field $F$ of positive characteristic $p$, $\theta \in H_p^{n-1}(F)$ and $\beta,\gamma \in F^\times$, we prove that if the symbols $\theta \wedge \frac{d \beta}{\beta}$ and $\theta \wedge \frac{d \gamma}{\gamma}$ in $H_p^n(F)$ share the same factors in $H_p^1(F)$ then the symbol $\theta \wedge \frac{d \beta}{\beta} \wedge \frac{d \gamma}{\gamma}$ in $H_p^{n+1}(F)$ is trivial. We conclude that when $p=2$, every two totally separably $(n-1)$-linked  $n$-fold quadratic Pfister forms are inseparably $(n-1)$-linked. We also describe how to construct non-isomorphic $n$-fold Pfister forms which are totally separably (or inseparably) $(n-1)$-linked, i.e. share all common $(n-1)$-fold  quadratic (or bilinear) Pfister factors. \end{abstract}

\end{frontmatter}

\section{Introduction}

Linkage of Pfister forms is a classical topic in quadratic form theory.
We say that two $n$-fold Pfister forms over a field $F$ are separably (inseparably, resp.) $m$-linked if there exists an $m$-fold quadratic (bilinear) Pfister form which is a common factor of both forms. When $\operatorname{char}(F) \neq 2$, there is no difference between quadratic and bilinear factors, so the terms coincide, and we simply say $m$-linked.

We say that two quadratic $n$-fold Pfister forms are totally separably (inseparably) $m$-linked if every quadratic (bilinear) $m$-fold Pfister factor of one of them is also a factor of the other. The following facts were proven in \cite{ChapmanDolphinLaghribi}:
\begin{itemize}
\item Two  $n$-fold quadratic Pfister forms can be totally separably 1-linked, inseparably 1-linked, or even both, without being isometric. (The special case of quaternion algebras over fields of characteristic not 2 was covered in \cite{GaribaldiSaltman}.)
\item Total separable 1-linkage and total inseparable 1-linkage of  $n$-fold  quadraticPfister forms are independent properties, i.e. do not imply each other.
\end{itemize} 

Clearly total separable (or inseparable) $(n-1)$-linkage implies (nontotal) separable (inseparable, resp.) $(n-1)$-linkage. 
It is known that inseparable $(n-1)$-linkage of quadratic $n$-fold Pfister forms implies separable $(n-1)$-linkage, but the converse is in general not true (see \cite{Lam:2002}, \cite{Faivre:thesis}, \cite{ChapmanGilatVishne:2017}, \cite{ChapmanMcKinnie:2018}, \cite{Chapman:2015} and \cite{BGL} for references). 
\begin{ques}
Does total separable $(n-1)$-linkage of  $n$-fold quadratic Pfister forms imply (nontotal) inseparable $(n-1)$-linkage?
\end{ques}
$$
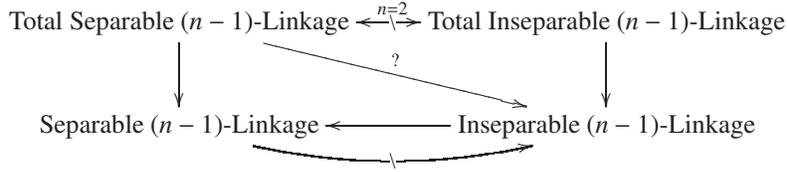
\xymatrix{
\text{Total Separable}\ (n-1)\text{-Linkage}\ar@{->}[d]\ar@{->}[dr]^{?}\ar@{<->}[r]|{\backslash}^{n=2}&\text{Total Inseparable}\ (n-1)\text{-Linkage}\ar@{->}[d]\\
\text{Separable}\ (n-1)\text{-Linkage}\ar@/_3ex/@{->}[r]|{\backslash}\ar@{<-}[r]&\text{Inseparable}\ (n-1)\text{-Linkage}
}$$
We answer this question in the positive in Section \ref{SectionInv}. We conclude it from deeper results on linkage of symbols in the Kato-Milne cohomology groups.

There are several other natural questions that arise in this setting:
\begin{ques}\label{Q1}
Do there exist totally separably (inseparably) $m$-linked quadratic $n$-fold Pfister forms which are not isometric for a given $m \in \{1,\dots,n-1\}$?
\end{ques}
\begin{ques}\label{Q2}
Over fields of characteristic 2, are total separable $m$-linkage and total inseparable $m$-linkage independent properties?
\end{ques}
\begin{ques}\label{Q3}
Given $1 \leq \ell < m \leq n-1$, are there totally separably (or inseparably) $\ell$-linked  $n$-fold  quadratic Pfister forms which are not totally separably (inseparably, resp.) $m$-linked?
\end{ques}

We answer Question \ref{Q1} in full generality (in the positive), and Question \ref{Q3} in the case of fields of characteristic not 2 and $m=n-1$ (see Section \ref{SectionTypes}). Question 1.3 was answered in the negative in \cite{ChapmanDolphinLaghribi} for $m=1$, but it remains open for arbitrary $m$.

\section{Preliminaries}

\subsection{Quadratic Forms}
For general reference on symmetric bilinear forms and quadratic forms see \cite{EKM}. 
Throughout, let $F$ be a field and  $V$ an  $F$-vector space. %A symmetric bilinear form on $V$ is a map $B : V \times V \rightarrow F$ satisfying $B(v,w)=B(w,v)$, $B(cv,w)=c B(v,w)$ and $B(v+w,t)=B(v,t)+B(w,t)$ for all $v,w,t \in V$ and $c \in F$.
%A symmetric bilinear form $B$ is degenerate if there exists a vector $v \in V\setminus\{0\}$ such that $B(v,w)=0$ for all $w \in V$.
%If such a vector does not exist, we say that $B$ is nondegenerate.
%Two symmetric bilinear  forms $B: V\times V \rightarrow F$ and $B' : W \times W\rightarrow F$ are isometric if there exists an isomorphism $M : V \rightarrow W$ such that $B(v,v')=B'(Mv,Mv')$ for all $v,v' \in V$.
%
A quadratic form over $F$ is a map $\varphi : V \rightarrow F$ 
such that $\varphi(av)= a^2\varphi(v)$ for all $a\in F$ and $v\in V$ and the map defined by  $B_\varphi(v,w)=\varphi(v+w)-\varphi(v)-\varphi(w)$ for all $v,w\in V$ is a bilinear form on $V$. The bilinear form $B_\varphi$ is called the polar form of $\varphi$ and is clearly symmetric.  
%given by a homogeneous polynomial of degree $2$, i.e. a map of the form $\varphi(u_1,\dots,u_n)=\sum_{i=1}^n \sum_{j=i}^n c_{i,j} u_i u_j$ for some $c_{i,j} \in F$ where $V$
%=\underbrace{F \times \dots \times F}_{n \enspace \text{times}}$
 %is an $n$-dimensional $F$-vector space.
  Two quadratic forms $\varphi : V \rightarrow F$ and $\psi : W \rightarrow F$ are isometric if there exists an isomorphism $M : V \rightarrow W$ such that $\varphi(v)=\psi(Mv)$ for all $v \in V$.
We are interested in the isometry classes of quadratic forms, so when we write $\varphi=\psi$ we actually mean that they are isometric.

%
%A quadratic form over $F$ is a map $\varphi : V \rightarrow F$ of the form $\varphi(u_1,\dots,u_n)=\sum_{i=1}^n \sum_{j=i}^n c_{i,j} u_i u_j$ for some $c_{i,j} \in F$ for some $n$-dimensional $F$-vector space $V$. Two quadratic forms $\varphi : V \rightarrow F$ and $\psi : W \rightarrow F$ are isometric  if there exists an isomorphism $M : V \rightarrow W$ for which $\varphi(v)=\psi(Mv)$ for all $v \in V$.
%We distinguish quadratic forms only up to isometry, so when we write $\varphi=\psi$ we mean that they are isometric.
%
%A symmetric bilinear form over $F$ is a map $B : V \times V \rightarrow F$ satisfying $B(v,w)=B(w,v), B(cv,w)=c B(v,w)$ and $B(v+w,t)=B(v,t)+B(w,t)$ for all $v,w,t \in V$ and $c \in F$ for some $n$-dimensional $F$-vector space $V$.
%A symmetric bilinear form $B$ is degenerate if there exists a vector $v \in V\setminus\{0\}$ such that $B(v,w)=0$ for all $w \in V$.
%If no such vector exists, we say that $B$ is nondegenerate.
%One can define isometry of symmetric bilinear forms in a similar way to quadratic forms. 

%Given a quadratic form $\varphi : V \rightarrow F$, one can consider the underlying symmetric bilinear form $B_\varphi : V \times V \rightarrow F$ defined by $B_\varphi(v,w)=\varphi(v+w)-\varphi(v)-\varphi(w)$. 
We say that $\varphi$ is singular if $B_\varphi$ is degenerate, and that $\varphi$ is nonsingular if $B_\varphi$ is nondegenerate.
If $F$ is of characteristic $2$, every nonsingular form $\varphi$ is even dimensional and can be written as 
$$\varphi=[\alpha_1,\beta_1] \perp \dots \perp [\alpha_n,\beta_n]$$
for some $\alpha_1,\dots,\beta_n \in F$, where $[\alpha,\beta]$ denotes the two-dimensional quadratic form $\psi(x,y)=\alpha x^2+xy+\beta y^2$ and $\perp$ denotes  the orthogonal sum of quadratic forms. If the characteristic of $F$ is different from $2$, symmetric bilinear forms and quadratic forms are equivalent objects, and we do not distinguish between them in this case. 
%
%We say that a quadratic form $\varphi : V \rightarrow F$ is isotropic if there exists a vector $v \in V\setminus\{0\}$ such that $\varphi(v)=0$.
%If no such vector exists, we say that $\varphi$ is anisotropic.
The unique nonsingular two-dimensional isotropic quadratic form is $\varmathbb{H}=[0,0]$, called the hyperbolic plane. 
A hyperbolic form is an orthogonal sum of hyperbolic planes. Any quadratic form $\varphi$ over $F$ decomposes into an orthogonal sum of a uniquely determined anisotropic quadratic form and a number of hyperbolic planes. The number of hyperbolic planes appearing in this decomposition is called the Witt index and denoted $i_W(\varphi)$.
%We say that two nonsingular quadratic forms are Witt equivalent if their orthogonal sum is a hyperbolic form.

We denote by $\langle \alpha_1,\dots,\alpha_n \rangle$ the diagonal bilinear form given by  
$(x,y)\mapsto \sum_{i=1}^n \alpha_ix_iy_i$. 
%$B(c_1 v_1+\dots+c_n v_n,d_1 v_1+\dots+d_n v_n)=\alpha_1 c_1 d_1+\dots+\alpha_n c_n d_n\,.$$
%Given two symmetric bilinear forms $B_1:V\times V\rightarrow F$ and $B_2:W\times W\rightarrow F$,   the tensor product of $B_1$ and $B_2$ denoted $B_1\otimes B_2$  
% is the unique $F$-bilinear map $B_1\otimes B_2:(V\otimes_F W)\times (V\otimes_F W)\rightarrow F$ such that 
%$(B_1\otimes B_2)\left( (v_1\otimes w_1), (v_2\otimes w_2)\right) =B_1(v_1,v_2)\cdot B_2(w_1,w_2) $
%for all $w_1,w_2\in W, v_1,v_2\in V$.
%
A bilinear $n$-fold Pfister form over $F$ is a symmetric bilinear form isometric to a bilinear form
$$\langle 1,\alpha_1\rangle \otimes \langle 1,\alpha_2\rangle \otimes \dots \otimes \langle 1,\alpha_n\rangle$$
for some $\alpha_1,\alpha_2,\dots,\alpha_n \in F^\times$.
We denote such a form by $\langle \! \langle \alpha_1,\alpha_2,\dots,\alpha_n \rangle \! \rangle$. By convention, the bilinear 0-fold Pfister form is $\langle 1 \rangle$. The $2$-fold Pfister  forms generate the fundamental  ideal $IF$ in the Witt ring of nondegenerate symmetric bilinear forms $WF$. Powers of $IF$ are denoted by $I^nF$, and are generated by $n$-fold Pfister forms respectively.

%The tensor product of  a diagonal bilinear form $B=\langle \alpha_1,\dots,\alpha_n \rangle$ and a quadratic form $\varphi$ is defined to be $B\otimes\psi=\alpha_1 \varphi \perp \dots \perp \alpha_n \varphi$.

  Let $B:V\times V\rightarrow F$ be a symmetric bilinear form over $F$ and $\varphi:W\rightarrow F$ be a quadratic form over $F$. We may define a  quadratic form $B\otimes \varphi:V\otimes_F W\rightarrow F$ determined by the rule that 
$( B\otimes \varphi) (v\otimes w)=  B(v,v) \cdot \varphi(w)$
for all $w\in W, v\in V$. We call this quadratic form  the tensor product of $B$ and $\varphi$. 
A quadratic $n$-fold Pfister form over $F$ is a tensor product of a bilinear $(n-1)$-fold Pfister form  $\langle \! \langle \alpha_1,\alpha_2,\dots,\alpha_{n-1} \rangle \! \rangle$ and a two-dimensional quadratic form  $[1,\beta]$ for some $\beta \in F$.
We denote such a form by $\langle \! \langle \alpha_1,\dots,\alpha_{n-1},\beta]\!]$. Quadratic $n$-fold Pfister forms are isotropic if and only if they are hyperbolic (see  \cite[(9.10)]{EKM}).
The $2$-fold quadratic Pfister forms generate the fundamental ideal, denoted $I_qF$ or $I_q^1 F$, of the Witt group of nonsingular quadratic forms.
Let $I_q^n F$ denote the subgroup generated by scalar multiples of quadratic $n$-fold Pfister forms.

%
%
%The group $W_q F=I_q F$ is generated by the forms $\varphi(u,v)=\alpha u^2+u v+\beta v^2$ for $\alpha,\beta \in F$, denoted by $[\alpha,\beta]$. We write $\langle \beta_1,\dots,\beta_n \rangle_b$ for the diagonal bilinear form $$B((v_1,\dots,v_n),(w_1,\dots,w_n))=\sum_{i=1}^n \beta_i v_i w_i$$ and $\langle \beta_1,\dots,\beta_n \rangle$ for the diagonal quadratic form $\varphi(v_1,\dots,v_n)=\sum_{i=1}^n \beta_i v_i^2$.
%
%The bilinear forms $\Pf{\beta} = \Qf{1,\beta}$ are called bilinear $1$-fold Pfister forms. These forms generate the fundamental ideal $IF$ of $WF$. Powers of $IF$ are denoted by $I^nF$. The tensor products $\Pf{\beta_1,\dots,\beta_n} = \Pf{\beta_1} \tensor \cdots \tensor \Pf{\beta_n}$ are called bilinear $n$-fold Pfister forms. 
%
%The quadratic form $[1,\alpha]$ is called a quadratic $1$-fold Pfister form, and denoted by $\MPf{\alpha}$. 
%For any quadratic form $\varphi$ and $\beta_1,\dots,\beta_n \in F^\times$, $\Qf{\beta_1,\dots,\beta_n} \otimes \varphi = \beta_1 \varphi \perp \dots \perp \beta_n \varphi$.
%For any integer $n \geq 2$, we define the quadratic $n$-fold Pfister form $\MPf{ \beta_1,\dots,\beta_{n-1},\alpha}$ as $\Pf{\beta_{1},\dots,\beta_{n-1}} \otimes \MPf{\alpha}$.
%A quadratic Pfister form is isotropic if and only if it is hyperbolic, and a bilinear Pfister form is isotropic if and only if it is metabolic.
%We define $I_q^n F$ to be group generated by the scalar multiples of quadratic $n$-fold Pfister forms.
Given a symmetric bilinear form $B$, we denote by $Q(B)$ the quadratic form given by the map  $v \mapsto B(v,v)$.

Let $\pi$ be an $n$-fold quadratic Pfister form over $F$. For  $m\in\{1,\ldots, n\}$,
we say an $m$-fold quadratic (resp.~bilinear) Pfister form $\psi$ (resp.~ $B$) is a factor of $\pi$ if there exists an $(n-m)$-fold bilinear  (reps.~quadratic) Pfister form $B'$  (resp.~$\psi'$) such that $\pi=B'\otimes \psi$ (resp.~$\pi=B\otimes \psi'$).

Let $\omega$ be an $n$-fold quadratic  Pfister form over $F$. We say $\pi$ and $\omega$ are separably  (resp.~inseparably) $m$-linked if there exists an $m$-fold quadratic (resp.~bilinear) Pfister form $\psi$  such that $\psi$ is a factor of both $\pi$ and $\omega$. We say $\pi$ and $\omega$ are totally separably (resp.~inseparably) $m$-linked if every quadratic (resp.~bilinear) $m$-fold Pfister form is a factor of $\pi$ if and only if it is a factor of $\omega$.  This terminology comes from the fact that  in characteristic $2$, the function fields of quadratic (resp.~bilinear) Pfister forms are separable (resp.~inseparable) extensions of the ground field.

\subsection{Kato-Milne Cohomology}

In this section, assume $F$ is a field of characteristic $p>0$. For $n>0$, the Kato-Milne Cohomology group $H_p^{n+1}(F)$ is defined to be the cokernel of the Artin-Schreier map
$$\wp : \Omega_F^n \rightarrow \Omega_F^n/d \Omega_F^{n-1}$$
$$\alpha \frac{d \beta_1}{\beta_1} \wedge \dots \wedge \frac{d \beta_n}{\beta_n} \mapsto (\alpha^p-\alpha) \frac{d \beta_1}{\beta_1} \wedge \dots \wedge \frac{d \beta_n}{\beta_n}.$$
We also fix $H_p^1(F)$  to be $F/\wp(F)$. The group $\nu_F(n)$ is defined to be the kernel of this map.
By \cite{BlochKato:1986}, $\nu_F(n) \cong K_n F/p K_n F$, with the isomorphism given by
$$\frac{d \beta_1}{\beta_1} \wedge \dots \wedge \frac{d \beta_n}{\beta_n} \mapsto \{\beta_1,\dots,\beta_n\}.$$
It is known that 
$H_p^2(F) \cong {_pBr(F)}$ and for $p=2$, $H_2^n(F) \cong I_q^n F/I_q^{n+1} F$.
The first isomorphism is given by the map
$$\alpha \frac{d \beta}{\beta} \mapsto [\alpha,\beta)_{p,F}$$
where $[\alpha,\beta)_{p,F}$ stands for the symbol $p$-algebra
$$F \langle x,y : x^p-x=\alpha, y^p=\beta, y x y^{-1}=x+1 \rangle.$$
The second isomorphism is given by the map
$$\alpha \frac{d \beta_1}{\beta_1} \wedge \dots \wedge \frac{d \beta_n}{\beta_n} \mapsto \langle \! \langle \beta_1,\dots,\beta_n,\alpha]\!].$$

We call the logarithmic differentials $\alpha \frac{d \beta_1}{\beta_1} \wedge \dots \wedge \frac{d \beta_{n-1}}{\beta_{n-1}}$ in $H_p^n(F)$ and $\frac{d \gamma_1}{\gamma_1} \wedge \dots \wedge \frac{d \gamma_m}{\gamma_m}$ in $\nu_F(m)$ ``symbols".
There is a natural map
$$H_p^n(F) \times \nu_F(m) \rightarrow H_p^{m+n}(F)$$
defined by the wedge product
$$\left(\alpha \frac{d \beta_1}{\beta_1} \wedge \dots \wedge \frac{d \beta_{n-1}}{\beta_{n-1}},\frac{d \gamma_1}{\gamma_1} \wedge \dots \wedge \frac{d \gamma_m}{\gamma_m}\right) \mapsto \alpha \frac{d \beta_1}{\beta_1} \wedge \dots \wedge \frac{d \beta_{n-1}}{\beta_{n-1}}\wedge \frac{d \gamma_1}{\gamma_1} \wedge \dots \wedge \frac{d \gamma_m}{\gamma_m}.$$

We define the linkage of symbols in an analogous manner to the linkage of Pfister forms. 
If a symbol $\omega$ in $H_p^{m+n}(F)$ is a wedge product $\theta \wedge \psi$ where $\theta$ is a symbol in $H_p^n(F)$ and $\psi$ is a symbol in $\nu_F(m)$, then $\theta$ and $\psi$ are called factors of $\omega$. We say that two symbols $\pi$ and $\omega$ are separably $k$-linked if they have a common factor in $H_p^k(F)$, and inseparably $k$-linked if they have a common factor in $\nu_F(k)$.
We say that two symbols $\pi$ and $\omega$ are totally separably $k$-linked if they share all factors in $H_p^k(F)$, and inseparably $k$-linked if they share all factor in $\nu_F(k)$.

\section{Separably $(n-1)$-linked Symbols in $H_p^n(F)$}\label{SectionInv}

In this section, assume  $F$ is a field of characteristic $p>0$.  One of the main goals is to show that total separable $(n-1)$-linkage implies inseparable  $(n-1)$-linkage for quadratic $n$-fold Pfister forms when $p=2$. 
%Recall first than an $F$-algebra is called a quadratic \'etale extension of $F$  if it is either a separable quadratic extension of $F$ or an $F$-algebra isomorphic to $F\times F$.

\begin{lem}\label{commonslot} For $\alpha\in F$ and $\beta \in F^\times$,  let $$t=  \alpha +\frac{(\alpha-\beta)}{\gamma}\,.$$
The symbol $p$-algebra $[\alpha,\gamma)_{p,F}$ contains the \'etale extension $F[x :x^p-x=t^p\gamma +\beta]$ of $F$. 
\end{lem}

\begin{proof}
Let $i$ and $j$ be a pair of generators of $[\alpha,\gamma)_{p,F}$ with $i^p-i=\alpha$, $j^p=\gamma$  and $jij^{-1}=i+1$.
Take  $x= i+tj +ij$ in $[\alpha,\gamma)_{p,F}$. Then $x^p-x$ is equal to $$\gamma \alpha^p+\gamma^{1-p} \alpha^p-\gamma^{1-p} \beta^p+\beta = t^p\gamma +\beta$$  by \cite[Lemma 3.1]{Chapman:2015}. Hence  the subalgebra $F[x]$ of  $[\alpha,\gamma)_{p,F}$ is as required.
\end{proof}

\begin{prop}\label{main}
Consider two separably $(n-1)$-linked symbols $$\pi=\alpha \frac{d \delta_1}{\delta_1} \wedge \dots \wedge \frac{d \delta_{n-2}}{\delta_{n-2}} \wedge \frac{d \beta}{\beta}\quad \textrm{ and }\quad\omega=\alpha \frac{d \delta_1}{\delta_1} \wedge \dots \wedge \frac{d \delta_{n-2}}{\delta_{n-2}} \wedge \frac{d \gamma}{\gamma}$$ in $H_p^n(F)$ and let $t=  \alpha +\frac{(\alpha-\beta)}{\gamma}$.
If $t^p \gamma +\beta$ is a factor in $H_p^1(F)$ of $\pi$, then the class of $\alpha \frac{d \delta_1}{\delta_1} \wedge \dots \wedge \frac{d \delta_{n-2}}{\delta_{n-2}} \wedge \frac{d \beta}{\beta} \wedge \frac{d\gamma}{\gamma}$ in $H_p^{n+1}(F)$ is trivial.
\end{prop}

\begin{proof} %Let  $t=\alpha+(\alpha-\beta)\gamma^{-1}$. 
Note first that if $t^p\gamma +\beta =0$ then $d\gamma\wedge d\beta=0$ and the result holds. Assume otherwise. 
The class of $ t^p\gamma +\beta$ in $H_p^1(F)$ is a factor of $\omega$ by \cref{commonslot}, so  it is a common factor of $\pi$ and $\omega$. 
We have  
$$\alpha \frac{d \beta}{\beta} \wedge \frac{d \gamma}{\gamma}=\alpha \frac{d (\beta \gamma^{-1})}{\beta \gamma^{-1}} \wedge \frac{d(t^p \gamma+\beta)}{t^p \gamma+\beta}$$ (see \cite[Lemma 5.1, (e)]{ChapmanMcKinnie:2018}).
Now, 
$$\alpha \frac{d \delta_1}{\delta_1} \wedge \dots \wedge \frac{d \delta_{n-2}}{\delta_{n-2}} \wedge \frac{d \beta}{\beta} \wedge \frac{d \gamma}{\gamma}=$$
$$\alpha \frac{d \delta_1}{\delta_1} \wedge \dots \wedge \frac{d \delta_{n-2}}{\delta_{n-2}} \wedge \frac{d \beta \gamma^{-1}}{\beta \gamma^{-1}} \wedge \frac{d (t^p \gamma+\beta)}{t^p \gamma+\beta}=$$
$$\alpha \frac{d \delta_1}{\delta_1} \wedge \dots \wedge \frac{d \delta_{n-2}}{\delta_{n-2}} \wedge \frac{d \beta}{\beta} \wedge \frac{d (t^p \gamma+\beta)}{t^p \gamma+\beta}-\alpha \frac{d \delta_1}{\delta_1} \wedge \dots \wedge \frac{d \delta_{n-2}}{\delta_{n-2}} \wedge \frac{d \gamma}{\gamma} \wedge \frac{d (t^p \gamma+\beta)}{t^p \gamma+\beta}.$$
Since $t^p \gamma+\beta$ is a factor in $H_p^1(F)$ of $\alpha \frac{d \delta_1}{\delta_1} \wedge \dots \wedge \frac{d \delta_{n-2}}{\delta_{n-2}} \wedge \frac{d \beta}{\beta}$, we have
$$\alpha \frac{d \delta_1}{\delta_1} \wedge \dots \wedge \frac{d \delta_{n-2}}{\delta_{n-2}} \wedge \frac{d \beta}{\beta} \wedge \frac{d (t^p \gamma+\beta)}{t^p \gamma+\beta}=(t^p \gamma+\beta) \tau \wedge \frac{d(t^p \gamma+\beta)}{t^p \gamma+\beta}$$
for some $\tau \in \nu_F(n-1)$.
As $(t^p \gamma+\beta) \frac{d(t^p \gamma+\beta)}{t^p \gamma+\beta}= d(t^p\gamma +\beta)$, it is  trivial  in $H_p^2(F)$. Hence  $\alpha \frac{d \delta_1}{\delta_1} \wedge \dots \wedge \frac{d \delta_{n-2}}{\delta_{n-2}} \wedge \frac{d \beta}{\beta} \wedge \frac{d (t^p \gamma+\beta)}{t^p \gamma+\beta}=0$.
Similarly, since  $t^p \gamma+\beta$ is a factor in $H_p^{1}(F)$ of $\alpha \frac{d\gamma}{\gamma}$, we have $\alpha \frac{d \delta_1}{\delta_1} \wedge \dots \wedge \frac{d \delta_{n-2}}{\delta_{n-2}} \wedge \frac{d \gamma}{\gamma} \wedge \frac{d (t^p \gamma+\beta)}{t^p \gamma+\beta}=0$.
Therefore  $\alpha \frac{d \delta_1}{\delta_1} \wedge \dots \wedge \frac{d \delta_{n-2}}{\delta_{n-2}} \wedge \frac{d \beta}{\beta} \wedge \frac{d \gamma}{\gamma}=0$ in $H_p^{n+1}(F)$ as required.
\end{proof}

\begin{cor}\label{maincor}
Let  $$\pi=\alpha \frac{d \delta_1}{\delta_1} \wedge \dots \wedge \frac{d \delta_{n-2}}{\delta_{n-2}} \wedge \frac{d \beta}{\beta} \quad \textrm{ and } \quad\omega=\alpha \frac{d \delta_1}{\delta_1} \wedge \dots \wedge \frac{d \delta_{n-2}}{\delta_{n-2}} \wedge \frac{d \gamma}{\gamma}$$ 
be  two separably $(n-1)$-linked symbols in $H_p^n(F)$.
If $\pi$ and $\omega$ are totally separably 1-linked then the class of $$\alpha \frac{d \delta_1}{\delta_1} \wedge \dots \wedge \frac{d \delta_{n-2}}{\delta_{n-2}} \wedge \frac{d \beta}{\beta} \wedge \frac{d\gamma}{\gamma}$$ in $H_p^{n+1}(F)$ is trivial.
\end{cor}

When $p=2$, by the identification of the symbols with quadratic $n$-fold Pfister forms, we obtain the following results:

\begin{prop} 
Assume $p=2$.
Let  $$\pi=\langle \! \langle \beta,\delta_{n-2},\dots,\delta_1,\alpha]\!]\quad \textrm{ and }\quad \omega=\langle \! \langle \gamma,\delta_{n-2},\dots,\delta_1,\alpha]\!]$$ 
be two separably $(n-1)$-linked  $n$-fold quadratic Pfister forms over $F$ and let $t=  \alpha +\frac{(\alpha-\beta)}{\gamma}$.
If the $1$-fold Pfister form $\langle \! \langle t^p\gamma +\beta]\!]$ is a factor of $\pi$, then $\langle \! \langle \beta,\gamma,\delta_{n-2},\dots,\delta_1,\alpha]\!]$ is trivial. In particular, $\pi$ and $\omega$ are inseparably $(n-1)$-linked.
\end{prop}

\begin{proof}
By \cite{ChapmanGilatVishne:2017}, the $(n+1)$-fold Pfister form $\langle \! \langle \beta,\gamma,\delta_{n-1},\dots,\delta_2,\alpha]\!]$ is trivial if and only if $\pi$ and $\omega$ are inseparably $(n-1)$-linked.
\end{proof}

%\begin{cor} 
%Given two separably $(n-1)$-linked quadratic $n$-fold Pfister forms $\pi=\langle \! \langle \beta,\delta_{n-1},\dots,\delta_2,\alpha]\!]$ and $\omega=\langle \! \langle \gamma,\delta_{n-1},\dots,\delta_2,\alpha]\!]$, if $\pi$ and $\omega$ are totally separably 1-linked then $\langle \! \langle \beta,\gamma,\delta_{n-1},\dots,\delta_2,\alpha]\!]$ is trivial. As a result, $\pi$ and $\omega$ are inseparably $(n-1)$-linked. {\color{red} Is this necessary to have as well as the previous and following corollaries?}
%\end{cor}

\begin{cor} Assume $p=2$. If a pair of separably $(n-1)$-linked  $n$-fold quadratic Pfister forms over $F$ are totally separably $1$-linked then they are also inseparably $(n-1)$-linked. In particular, if a pair of $n$-fold quadratic Pfister forms  over $F$ are totally separably $(n-1)$-linked then they are  inseparably $(n-1)$-linked.
\end{cor}

\begin{rem}
A similar  result to \cref{maincor} holds more straight-forwardly for Milnor $K$-groups. 
Let $p$ be a prime integer, $n$ a positive integer and $F$ an arbitrary field. If $p=2$ then further assume that $\sqrt{-1}\in F$. %We define $m$-linkage (or total $m$-linkage) of symbols in $K_n F/p K_n F$ by sharing one (all, resp.) factors in $K_m F/p K_m F$. 
Then the following is trivial:
\begin{itemize}
\item If $\{\alpha\} \cup \theta \in K_n F/pK_n F$ has $\{\beta\}$ as a factor in $K_1 F/p K_1 F$ then $\{\alpha,\beta\} \cup \theta=0$ in $K_{n+1} F/p K_{n+1} F$.
\item Therefore, if $\{\alpha\} \cup \theta$ and $\{\beta\} \cup \theta$ in $K_n F/pK_n F$ are totally 1-linked then $\{\alpha,\beta\} \cup \theta=0$ in $K_{n+1} F/p K_{n+1} F$.
\end{itemize}
\end{rem}

\begin{rem} The result analogous  to \cref{maincor} for  inseparable linkage is also straight-forward.
For fields $F$ of characteristic $p>0$ and positive integer $n$,
\begin{itemize}
\item If $\omega \wedge \frac{d\beta}{\beta} \in H_p^n F$ has $\frac{d \gamma}{\gamma}$  in $\nu_F(1)$ as a factor then $\omega \wedge \frac{d \beta}{\beta} \wedge \frac{d \gamma}{\gamma}=0$ in $H_p^{n+1}(F)$.
\item Therefore, if $\omega \wedge \frac{d\beta}{\beta}$ and $\omega \wedge \frac{d\gamma}{\gamma}$ in $H_p^n F$ are totally inseparably 1-linked then $\omega \wedge \frac{d \beta}{\beta} \wedge \frac{d \gamma}{\gamma}=0$ in $H_p^{n+1}(F)$.
\end{itemize}
\end{rem}

\section{Totally Linked Quadratic Pfister Forms}\label{SectionTypes}
%
%\begin{defn}
%Let $F$ be a field and $n$ be an integer $\geq 2$.
%Two quadratic $n$-fold Pfister forms over $F$, $\varphi_1$ and $\varphi_2$, are said to be totally $m$-linked if they share all quadratic $m$-fold Pfister factors.
%\end{defn}

In \cite{ChapmanDolphinLaghribi} we considered whether total $1$-linkage of Pfister forms implied isometry. In general it does not. In this section, we consider whether total  $m$-linkage implies isometry.

\begin{lem}\label{wittindex} Let $n$ be an integer $\geq 2$ and   $m\in \{1,\dots,n-1\}$.
Let $\varphi$ be an  $n$-fold quadratic Pfister form, $\pi$ an  $m$-fold  quadratic Pfister form and $\theta$ an  $(m-1)$-fold quadratic Pfister form.
Assume $\theta$ is a common factor of $\varphi$ and $\pi$.
Then $i_W(\varphi \perp -\pi)=2^m$ if and only if $\pi$ is a factor of $\varphi$. Otherwise, $i_W(\varphi \perp -\pi)=2^{m-1}$.
\end{lem}

\begin{proof}
By \cite[Corollary 24.3]{EKM}, $i_W(\varphi \perp -\pi)$ must be a power of 2.
Since $\theta$ is a factor of $\varphi$, $i_W(\varphi \perp -\pi) \geq 2^{m-1}$. Therefore, the only other possible value is $2^m$, in which case $\pi$ is a subform of $\varphi$ and therefore a factor of $\varphi$.
\end{proof}

\begin{lem}\label{func}
Let $n$ be an integer $\geq 2$  and   $m\in \{1,\dots,n-1\}$.
Let $\varphi$ and $\psi$ be two non-hyperbolic, separably $(n-1)$-linked and totally separably $m$-linked  $n$-fold quadratic Pfister forms over $F$.  Let $\pi$ be an $(m+1)$-fold quadratic Pfister form such that $\pi$ is a factor of $\varphi$ but not $\psi$. Then there exists a field extension $L$ such that $\pi_L$ is a factor of $\psi_L$ and $\varphi_L$ and $\psi_L$ are neither isometric nor hyperbolic. 
\end{lem}
\begin{proof}
Since $\varphi$ and $\psi$ are separably $(n-1)$-linked, the form $\varphi \perp \psi$ is congruent mod $I_q^{n+1} F$ to some anisotropic $n$-fold Pfister form $\phi$. Let $\pi_0$ be an  $m$-fold quadratic Pfister factor of $\pi$. Since $\varphi$ and $\psi$ are totally separably  $m$-linked, $\pi_0$ is a common factor of both forms.

By  \cref{wittindex},  $\psi \perp -\pi$ is Witt equivalent to some anisotropic $2^n$-dimensional form $\theta$. Write $L=F(\theta)$ for the function field of $\theta$ over $F$. If one of the  forms $\varphi_L, \psi_L$ and $\phi_L$ were  hyperbolic, then $\theta$ would be similar to a  subform of the form by \cite[Corollary 22.5]{EKM}. However,  since the forms are of the same dimension, this would imply that  $\theta$ is similar to an $n$-fold Pfister form. This is impossible because the $m$th cohomological invariant of $\theta$ is nontrivial.  It follows that $\psi_L$ and $\varphi_L$ are not isometric as $\phi_L$ is not hyperbolic.
\end{proof}

\begin{thm}\label{goup}
Let $n$ be an integer $\geq 2$  and   $m\in \{1,\dots,n-1\}$.
Let $\varphi$ and $\psi$ be two non-hyperbolic, separably $(n-1)$-linked and totally separably $m$-linked quadratic $n$-fold Pfister forms over $F$. Then there exists a field extension $K$ of $F$ such that $\varphi_K$ and $\psi_K$ are totally separably $(m+1)$-linked but not isometric nor hyperbolic.
\end{thm}
\begin{proof}
 Let $S$ be the set of $(m+1)$-fold quadratic Pfister forms $\pi$ over $F$ such that $\pi$ is a factor of $\varphi$ but not of $\psi$.  
Then  as in \cref{func}, for each $\pi\in S$,  there exists a $2^n$-dimensional  quadratic form $\theta$ such that $\theta$ is Witt equivalent to $\pi\perp   \psi$. Let 
 $F_0$ be the compositum of the function fields of all such $\theta$. 
 
 Similarly, let $T$ be the set of $(m+1)$-fold quadratic Pfister forms $\pi'$ over $F$ such that $\pi$ is a factor of $\psi$ but not of $\varphi$.  
Again, as in \cref{func}, for each $\pi'\in T$,  there exists a $2^n$-dimensional  quadratic form $\theta'$ such that $\theta'$ is Witt equivalent to $\pi'\perp   \varphi$. Let $F'_0$ be the field compositum of the function fields of all such $\theta'$. 

Let $K_0$ be the compositum of $F_0$ and $F'_0$. Then by \cref{func}, $\varphi_{K_0}$ and $\psi_{K_0}$ are not isometric nor hyperbolic, and every $(m+1)$-fold Pfister form over $K_0$ defined over $F$ is a factor of $\varphi_{K_0}$ if and only if it is a factor of $\psi_{K_0}$.
Using this construction inductively, we obtain the required field extension $K/F$.
\end{proof}

\begin{lem}\label{subform}
Assume  $\operatorname{char}(F)=2$. Let $\theta$ be an anisotropic $n$-fold bilinear Pfister form over $F$ and let $\theta'$ be an $(n-1)$-fold bilinear Pfister form factor of $\theta$. Let $\beta$ be an element represented by $\theta$ but not by $\theta'$. Then $Q(\theta)= Q(\langle\!\langle \beta \rangle \! \rangle \otimes \theta')$.
\end{lem}
\begin{proof}
This follows easily from \cite[Proposition 10.4]{EKM}.
\end{proof}

\begin{lem}\label{quasi}
Assume  $\operatorname{char}(F)=2$. Let $n$ be an integer $\geq 2$ and $m\in \{1,\dots,n-1\}$.
Let $\varphi$ be an anisotropic  $n$-fold quadratic  Pfister form, $\pi$ an anisotropic  $m$-fold  bilinear Pfister  form and $\theta$ an  $(m-1)$-fold bilinear Pfister form.
Assume $\theta$ is a common factor of $\varphi$ and $\pi$.
Then $i_W(\varphi \perp Q(\pi))=2^m$ if and only if $\pi$ is a factor of $\varphi$. Otherwise, $i_W(\varphi \perp Q(\pi))=2^{m-1}$.
\end{lem}

\begin{proof}
Since $\theta$ is a factor of $\varphi$, there exists a  $2^{n-m}-1$ dimensional  bilinear form $b$ and a quadratic $1$-fold Pfister form  $\rho$ such that
 $$\varphi= \theta \otimes (\qf{1}\perp b )\otimes \rho\,.$$
Set $\tau =\theta\otimes  b\otimes \rho$ (note that $b\otimes \rho$ is the so-called pure part of the quadratic Pfister form $(\qf{1}\perp b)\otimes \rho$, see \cite[p.66]{EKM}). Then we have that  $$\varphi \perp Q(\theta)=2^{m-1} \otimes \varmathbb{H} \perp \tau \perp Q(\theta)\,.$$  Note  that   $\tau \perp Q(\theta)$ is anisotropic as $\varphi$ is aniostropic.

Suppose $\tau \perp Q(\pi)$ is isotropic.
Since $\tau \perp Q(\theta)$ is anisotropic, there exists an element  $\beta$ represented by $\tau$ and  $Q(\pi)$ but not by $Q(\theta)$.
As $\beta$ is represented by    $Q(\pi)$ but not by $Q(\theta)$, it follows from  \cref{subform} that $Q(\langle\!\langle \beta \rangle \! \rangle \otimes \theta) = Q(\pi)$.
As $\beta$ is represented by    $Q(\tau)$ but not by $Q(\theta)$, it follows from \cite[Proposition 15.7]{EKM} that $\langle \! \langle \beta \rangle \! \rangle \otimes \theta$ is a factor of $\varphi$.
In particular,  $\varphi$ becomes isotropic over the function field of $\pi$. Hence $\pi$ is a factor of $\varphi$ by \cite[(1.4)]{Lag:2005} and repeated use of \cite[(15.6)]{EKM}.
\end{proof}

\begin{thm}\label{goupins}
Assume $\operatorname{char}(F)=2$ and $n$ be an integer $\geq 2$ and $m\in \{1,\dots,n-1\}$.
Let $\varphi$ and $\psi$ be two non-hyperbolic, separably $(n-1)$-linked and totally inseparably $m$-linked  $n$-fold quadratic Pfister forms over $F$. Then there exists a field extension $K$ of $F$ such that $\varphi_K$ and $\psi_K$ are totally inseparably $(m+1)$-linked but neither  isometric nor hyperbolic.
\end{thm}

\begin{proof} The result follows from \cref{quasi} in a similar way to \cref{goup}.
%Suppose $\pi$ is a bilinear $(m+1)$-fold Pfister factor of $\varphi$ but not of $\psi$. The form $\varphi \perp \psi$ is congruent mod $I_q^{n+1} F$ to some anisotropic $n$-fold Pfister form $\phi$. Let $\pi_0$ be a bilinear $m$-fold Pfister factor of $\pi$. Since $\varphi$ and $\psi$ are totally $m$-linked, $\pi_0$ is a common factor of both forms. Then $\psi \perp Q(\pi)$ is Witt equivalent to some anisotropic $2^n$-dimensional form $\theta$. Write $L=F(\theta)$ for the function field of $\theta$ over $F$. The forms $\varphi, \psi$ and $\phi$ are not split by $L$ because if one of them were, $\theta$ would be dominated by it, but since the forms are of the same dimension, $\theta$ would be similar to an $n$-fold Pfister form, which is impossible because $\theta$ is singular.
\end{proof}

Using Theorems \ref{goup} and \ref{goupins}, one can construct examples of non-isometric  pairs of  $n$-fold quadratic Pfister forms which are totally separably (or inseparably, or both) $m$-linked for any $m \in \{1,\dots,n-1\}$.

\begin{exmpl}
Start with the field $F=\mathbb{F}_2(\!(x_1)\!)\dots(\!(x_{n+1})\!)$ of iterated Laurent series in $n+1$ indeterminates over $\mathbb{F}_2$. The forms $$\varphi=\langle \! \langle x_1,\dots x_{n-1},x_1\cdot\ldots\cdot x_n]\!]\quad  \mathrm{ and }\quad  \psi=\langle \! \langle x_2,\dots x_{n},x_2\cdot\ldots\cdot x_{n+1}]\!]$$ over $F$ are $(n-1)$-linked but not totally 1-linked (see \cite[Sections 9\&10]{ChapmanGilatVishne:2017}). By iterating Theorem \ref{goup} (or \ref{goupins}) $m$ times, we end up with a field $K$ over which $\varphi_K$ and $\psi_K$ are totally $m$-linked, but neither isometric  nor hyperbolic.
\end{exmpl}

\begin{ques}
When $\operatorname{char}(F)=p$, does there exist a similar process that extends two non-equal, nontrivial  separably $(n-1)$ symbols in $H_p^n(F)$ to two  non-equal, nontrivial  totally separably (or inseparably) $(n-1)$-linked symbols?
\end{ques}
%
%\begin{defn}
%For two quadratic $n$-fold Pfister forms $\varphi_1$ and $\varphi_2$, their linkage number, denoted by $\operatorname{ln}(\varphi_1,\varphi_2)$, as defined in \cite[Definition 3.3]{Hoffmann:1996}, is the maximal $m$ for which $\varphi_1$ and $\varphi_2$ are $m$-linked.
%\end{defn}

\begin{thm}\label{control}
Assume  $\operatorname{char}(F) \neq 2$ and $n$ be an integer $\geq 3$ and $m\in \{1,\dots,n-3\}$.
Let $\varphi$ and $\psi$ be two non-hyperbolic  $(n-1)$-linked and totally $m$-linked  $n$-fold quadratic Pfister forms over $F$. Assume there exists an $(n-1)$-fold quadratic Pfister form $\omega$ such that 
\begin{enumerate}
\item[(a)] $\omega$ is a factor of $\varphi$,
\item[(b)] $\omega$ is not a factor of $\psi$,
\item[(c)] there exists an $(n-2)$-fold quadratic Pfister form that is a factor of both $\omega$ and $\psi$.
\end{enumerate}
Then there exists a field extension $K$ of $F$ such that $\varphi_K$ and $\psi_K$ are totally $(m+1)$-linked but not totally $(n-1)$-linked nor hyperbolic.
\end{thm}

\begin{proof}
This is essentially the same proof as in Theorem \ref{goup}.
Here we just need to note the following:
The form $\psi \perp -\omega$ is Witt equivalent to some anisotropic $2^n$ dimensional form. The latter remains anisotropic under scalar extension to $L$ by  \cite[Theorem 5.4]{Hoffmann:1996}. Therefore $\omega_L$ is not a subform of $\psi_L$, and that completes the proof.
\end{proof}

\begin{exmpl}
Start with the field $F=\mathbb{C}(\!(x_1)\!)\dots(\!(x_{n+1})\!)$ of iterated Laurent series in $n+1$ indeterminates over $\mathbb{C}$. The forms $\varphi=\langle\!\langle x_1,\dots,x_n\rangle\!\rangle$ and $\psi=\langle\!\langle x_2,\dots,x_{n+1}\rangle\!\rangle$ over $F$ are $(n-1)$-linked but not totally 1-linked. By iterating Theorem \ref{goup} $m$ times, we end up with a field $K$ over which $\varphi_K$ and $\psi_K$ are totally $m$-linked, but neither hyperbolic nor isometric. If $m \leq n-2$ then by Theorem \ref{control}, $\varphi_K$ and $\psi_K$ are also not totally $(n-1)$-linked.
\end{exmpl}

\section*{Acknowledgements}
The authors thank Jean-Pierre Tignol for his comments on the manuscript. We also thank the annoymous referee for a careful reading of the manuscript and helpful suggestions for its improvement.

\bibliographystyle{abbrv}

\begin{thebibliography}{10}

\bibitem{BGL}
C.~Beli, P.~Gille, and T.-Y.~Lee.
\newblock Examples of algebraic groups of type {$G_2$} having the same maximal tori.
\newblock {\em Proc. Steklov Inst. Math.}, 292(1):10--19, 2016.

\bibitem{BlochKato:1986}
S.~Bloch and K.~Kato.
\newblock {$p$}-adic \'etale cohomology.
\newblock {\em Inst. Hautes \'Etudes Sci. Publ. Math.}, 63:107--152, 1986.

\bibitem{Chapman:2015}
A.~Chapman.
\newblock Common subfields of {$p$}-algebras of prime degree.
\newblock {\em Bull. Belg. Math. Soc. Simon Stevin}, 22(4):683--686, 2015.

\bibitem{ChapmanDolphinLaghribi}
A.~Chapman, A.~Dolphin, and A.~Laghribi.
\newblock Total linkage of quaternion algebras and {P}fister forms in
  characteristic two.
\newblock {\em J. Pure Appl. Algebra}, 220(11):3676--3691, 2016.

\bibitem{ChapmanGilatVishne:2017}
A.~Chapman, S.~Gilat, and U.~Vishne.
\newblock Linkage of quadratic {P}fister forms.
\newblock {\em Comm. Algebra}, 45(12):5212--5226, 2017.

\bibitem{ChapmanMcKinnie:2018}
A.~Chapman and K.~McKinnie.
\newblock Kato-{M}ilne cohomology and polynomial forms.
\newblock {\em J. Pure Appl. Algebra}, 222(11):3547--3559, 2018.


\bibitem{EKM}
R.~Elman, N.~Karpenko, and A.~Merkurjev.
\newblock {\em The algebraic and geometric theory of quadratic forms},
  volume~56 of {\em American Mathematical Society Colloquium Publications}.
\newblock American Mathematical Society, Providence, RI, 2008.

\bibitem{Faivre:thesis}
F.~Faivre.
\newblock {\em Liaison des formes de {P}fister et corps de fonctions de
  quadriques en caract\'eristique $2$}.
\newblock 2006.
\newblock Thesis (Ph.D.)--Universit\'e de Franche-Comt\'e.

\bibitem{GaribaldiSaltman}
S.~Garibaldi and D.~J. Saltman.
\newblock Quaternion algebras with the same subfields.
\newblock In {\em Quadratic forms, linear algebraic groups, and cohomology},
  volume~18 of {\em Dev. Math.}, pages 225--238. Springer, New York, 2010.

\bibitem{Hoffmann:1996}
D.~W. Hoffmann.
\newblock Twisted {P}fister forms.
\newblock {\em Doc. Math.}, 1:No. 03, 67--102, 1996.



\bibitem{Lag:2005}
A.~Laghribi.
\newblock Witt kernels of function field extensions in characteristic $2$.
\newblock {\em J.Pure Appl. Algebra}, 199, 167--182, 2005.


\bibitem{Lam:2002}
T.~Y. Lam.
\newblock On the linkage of quaternion algebras.
\newblock {\em Bull. Belg. Math. Soc. Simon Stevin}, 9(3):415--418, 2002.



\end{thebibliography}

\end{document}